\documentclass[12pt]{amsart}
\usepackage{amsmath,amssymb,amscd,amsthm}
\usepackage{mathtools}
\usepackage{stmaryrd}
\usepackage[colorlinks,allcolors=blue]{hyperref}
\usepackage[left=2cm,right=2cm,top=2cm,bottom=2cm,includeheadfoot]{geometry}
\usepackage{ytableau}
\usepackage{scrextend}
\changefontsizes[14pt]{13.5pt}

\makeatletter
\@namedef{subjclassname@2010}{%
  \textup{2010} Mathematics Subject Classification}
\makeatother

\newtheorem{theorem}{Theorem}
\newtheorem{proposition}{Proposition}

\newtheorem {lemma}{Lemma}

\theoremstyle{definition}

\newtheorem{example}{Example}

\DeclareMathOperator\sgn{sgn}

\usepackage{stix}

\begin{document}

\baselineskip=17pt

\title[]{A characterization of the algebraic degree\\ in semidefinite programming}

\author{Dang Tuan Hiep}
\address{Faculty of Mathematics and Computer Science, Da Lat University, Lam Dong, Vietnam}
\email{hiepdt@dlu.edu.vn}

\author{Nguyen Thi Ngoc Giao}
\address{Faculty of Advanced Science and Technology, University of Science and Technology - The University of Da Nang, 54 Nguyen Luong Bang, Da Nang, Vietnam}
\email{ngocgiao185@gmail.com}

\author{Nguyen Thi Mai Van}
\address{Faculty of Mathematics and Statistics, Quy Nhon University, Binh Dinh, Vietnam}
\email{nguyenthimaivan@qnu.edu.vn}

\begin{abstract}
In this article, we show that the algebraic degree in semidefinite programming can be expressed in terms of the coefficient of a certain monomial in a doubly symmetric polynomial. This characterization of the algebraic degree allows us to use the theory of symmetric polynomials to obtain many interesting results of Nie, Ranestad and Sturmfels in a simpler way.
\end{abstract}



\date{\today}

\maketitle
\tableofcontents
\section{Introduction}

Consider the semidefinite programming (SDP) problem in the form 
\begin{equation}\label{SDP}
  \text{maximize trace}(B \cdot Y) \text{ subject to } Y \in \mathcal U \text{ and } Y \succeq 0,
\end{equation}
where $B$ is a real symmetric $n\times n$-matrix, $\mathcal U$ is a $m$-dimensional affine subspace in the $\binom{n+1}{2}$-dimensional space of real $n \times n$-symmetric matrices, and $Y \succeq 0$ means that $Y$ is positive semidefinite. We know that the coordinates of the optimal solution are the roots of some univariate polynomials. If the data are generic, then the degree of these polynomials depends only on the rank $r$ of the optimal solution. This is what we call the {\it algebraic degree} $\delta(m,n,r)$ in the semidefinite programming \eqref{SDP}. Note that the algebraic degree $\delta(m,n,r)$ is well-defined only if the triple $(m,n,r)$ satisfies {\it Pataki's inequalities} (see \cite[Proposition 5]{NRS}), that is
\begin{equation}\label{pataki}
  \binom{n-r+1}{2} \leq m \leq \binom{n+1}{2} - \binom{r+1}{2}.
\end{equation}

In \cite{NRS}, Nie, Ranestad and Sturmfels first introduced and showed that the algebraic degree $\delta(m,n,r)$ in semidefinite programming coincides with a degree of a dual variety (see \cite[Theorem 13]{NRS}), by using methods from complex algebraic geometry. In particular, one of their main results was a collection of explicit formulas of the algebraic degree $\delta(m,n,r)$ for various special values of $m, n, r$ (see \cite[Theorem 11]{NRS}), which was obtained via computing the Euler numbers of smooth varieties, the degrees of determinantal varieties, and so on. Moreover, they also gave a general formula of the algebraic degree $\delta(m,n,r)$ which was conjectured to hold arbitrary values of $m,n$ and $r$ (see \cite[Theorem 19, Conjecture 21]{NRS}). After that, Graf von Bothmer and Ranestad showed that the algebraic degree $\delta(m,n,r)$ can be computed as an intersection number of Segre classes of the second symmetric powers of universal bundles on the Grassmannian $G(r,n)$ (see \cite[Proposition 4.1]{BR2009}). Furthermore, by using desingularization of the conormal variety of the variety of symmetric matrices of rank less than $r$, they gave a general formula for $\delta(m,n,r)$ in terms of a function on subsequences of $\{1,2,\ldots,n\}$ (see \cite[Theorem 1.1]{BR2009}). The first author drived another formula using minors and hence also readily computable, following a totally different technique from equivariant cohomology (see \cite[Theorem 1]{H1}). Recently, the study of the algebraic degree and related problems have received more and more attention. For instance, Manivel et. al. confirmed \cite[Conjecture 21]{NRS}, providing another formula for $\delta(m, n, r)$ and another proof of the polynomiality theorem (see \cite{MLMSV} and references therein).

In this paper, we provide a combinatorial approach for calculating the algebraic degrees using doubly symmetric polynomials. More concretely, we show that the algebraic degree $\delta(m,n,r)$ can be expressed in terms of the coefficient of a certain monomial in a doubly symmetric polynomial (see Theorem \ref{main}). By this characterization, the duality relation of the algebraic degree $\delta(m,n,r)$ can be easily obtained (see Proposition \ref{duality_relation}). As an application, we apply this characterization to reprove many interesting results of Nie, Ranested and Sturmfels (see Propositions \ref{n-1}, \ref{3_n-2} and \ref{4_n-2}) via some facts about symmetric polynomials (see Propositions \ref{key_lm} and \ref{key_lm2}). On the other hands, our main result, Theorem \ref{main}, can be viewed as a corollary of the first author’s recent works \cite[Theorem 1]{H1} and \cite[Theorem 1.2]{H2}. As a short digression, we also provide a self-contained proof of \cite[Theorem 1.2]{H2} (see Proposition 1), which is inspired by a result due to Don Zagier in the appendix of \cite[Proposition A.1]{GMZ}.

From now on, for any tuple $(m,n,r)$ satisfying Pataki's inequalities \eqref{pataki}, we will use the following notations
\begin{equation}\label{kl}
\mathscr k = m - \binom{n-r+1}{2} \text{ and } \mathscr l = \binom{n+1}{2} - \binom{r+1}{2} - m.
\end{equation}
Note that $\mathscr k + \mathscr l = r(n-r)$. Throughout, we work on $\mathbb Q[x_1,\ldots,x_r,y_1,\ldots,y_{n-r}]$ the polynomial ring over $\mathbb Q$ in $n$ variables $x_1,\ldots,x_r,y_1,\ldots,y_{n-r}$, unless we explicitly indicate otherwise. Let us consider the following sets
\begin{equation}\label{sum_variables}
\mathcal{X} := \left\{ x_i+x_j \mid 1 \leq i \leq j \leq r \right\} \text{ and } \mathcal{Y} := \left\{ y_i+y_j \mid 1 \leq i \leq j \leq n-r \right\}.
\end{equation}
Clearly, the cardinality of $\mathcal{X}$ and $\mathcal{Y}$ are respectively $\sharp \mathcal{X} = \binom{r+1}{2}$ and $\sharp \mathcal{Y} = \binom{n-r+1}{2}$. 

\begin{theorem}\label{main}
The algebraic degree
$$\delta(m,n,r) = (-1)^{\mathscr k}\frac{c(m,n,r)}{r!(n-r)!},$$
where $c(m,n,r)$ is the coefficient of the monomial $x_1^{n-1}\cdots x_r^{n-1}y_1^{n-1}\cdots y_{n-r}^{n-1}$ in the polynomial 
\[h_{\mathscr l}(\mathcal{X})h_{\mathscr k}(\mathcal{Y}) \prod_{j \neq i}(x_i-x_j)\prod_{j \neq i}(y_i-y_j)\prod_{i=1}^{n-r}\prod_{j=1}^r(y_i - x_j),\]
where $h_{\mathscr l}(\mathcal{X})$ (resp. $h_{\mathscr k}(\mathcal{Y})$) is the complete homogeneous symmetric polynomial of degree $\mathscr l$ (resp. $\mathscr k$), in $\binom{r+1}{2}$ (resp. $\binom{n-r+1}{2}$) variables which are the elements of the sets $\mathcal{X}$ (resp. $\mathcal{Y}$). 
\end{theorem}

The rest of the paper is organized as follows: In Section 2, we recall an identity involving doubly symmetric polynomials and prove the main result. In Section 3, we present interesting results involving symmetric polynomials. In the final section, we use our characterization to reprove some known results of Nie-Ranestad-Sturmfels.

\section{Proof of the main result}

In this section, we present an identity involving doubly symmetric polynomials which appearing in \cite[Theorem 1.2]{H2} without proof. In which, the author claims that the proof of the identity is very similar to that of \cite[Theorem 1.1]{H2}, so it was omitted. We here present the proof using an effective argument which is inspired by a result due to Don Zagier in the appendix of \cite[Proposition A.1]{GMZ} (see also \cite[Lemma 2]{HT}). 

\begin{lemma}\label{Lem2}
Let $Q_1(x), \ldots, Q_n(x) \in \mathbb Q[x]$ be monic polynomials of degree $d_1+1, \ldots, d_n+1$ with distinct roots and 
$F(x_1,\ldots,x_r,y_1,\ldots,y_{n-r}) \in \mathbb Q[x_1,\ldots,x_r,y_1,\ldots,y_{n-r}]$ a polynomial in $n$ variables of degree not greater than $d_1+\ldots+d_n$. Then, the expression
\begin{equation}\label{eq1}
    \sum_{Q_1(\alpha_1)=\cdots=Q_n(\alpha_n)=0}\frac{F(\alpha_1, \ldots,\alpha_n)}{Q'_1(\alpha_1)\cdots Q'_n(\alpha_n)}
\end{equation}
is independent of all the $Q_i$ and is equal to the coefficient of the monomial $x_1^{d_1}\cdots x_r^{d_r}y_1^{d_{r+1}}\cdots y_{n-r}^{d_n}$ in $F(x_1, \ldots,x_r,y_1,\ldots,y_{n-r})$.
\end{lemma}
\begin{proof}
By linearity, it is enough to consider monomials \[F(x_1,\ldots,x_r,y_1,\ldots,y_{n-r}) = x_1^{a_1}\cdots x_r^{a_r}y_1^{a_{r+1}}\cdots y_{n-r}^{a_n},\]
where $a_1+\cdots +a_n \leq d_1+\cdots +d_n$. The expression \eqref{eq1} factors as
\[\bigg( \sum_{Q_1(\alpha_1) =0}\frac{\alpha_1^{a_1}}{Q'_1(\alpha_1)} \bigg)\cdots\bigg( \sum_{Q_n(\alpha_n) =0}\frac{\alpha_n^{a_n}}{Q'_n(\alpha_n)} \bigg).\]
For each $i \in [n]$, suppose that $\gamma_{i0},\gamma_{i1}, \ldots,\gamma_{id_i}$ be $d_i+1$ distinct roots of $Q_i(x)$, then
\[\sum_{Q_i(\alpha_i)=0}\frac{\alpha_i^{a_i}}{Q_i'(\alpha_i)}=\sum_{j=0}^{d_i}\frac{\gamma_{ij}^{a_i}}{\prod_{k\neq j}(\gamma_{ij}-\gamma_{ik})}.\]
By the Lagrange interpolation formula for $x^{a_i}$ at $\gamma_{i0},\gamma_{i1}, \ldots,\gamma_{id_i}$, we have
\[x^{a_i} = \sum_{j=0}^{d_i}\frac{\gamma_{ij}^{a_i}\prod_{k\neq j}(x-\gamma_{ik})}{\prod_{k\neq j}(\gamma_{ij}-\gamma_{ik})}.\]
This implies that
\begin{eqnarray*}
\sum_{j=0}^{d_i}\frac{\gamma_{ij}^{a_i}}{\prod_{k\neq j}(\gamma_{ij}-\gamma_{ik})}= \textrm{ the coefficient of } x^{d_i} =
\begin{cases}
0 &\textrm{ if } \quad 0 \leq a_i < d_i,\\
1 &\textrm{ if } \quad a_i=d_i.
\end{cases}
\end{eqnarray*}
Hence, if $a_i = d_i$ for all $i$, then the product will be equal to $1$, else the product will be equal to $0$. The lemma follows.
\end{proof}

Recall that a polynomial $P(x_1,\ldots,x_r,y_1,\ldots,y_{n-r}) \in \mathbb Q[x_1, \ldots,x_r,y_1,\ldots,y_{n-r}]$ is said to be \emph{doubly symmetric} if for any permutation $\sigma \in \mathcal{S}_r$ and for any  permutation $\theta \in \mathcal{S}_{n-r}$, one has
$$P(x_{\sigma(1)},\ldots,x_{\sigma(r)},y_{\theta(1)},\ldots,y_{\theta(n-r)}) = P(x_1, \ldots,x_r,y_1,\ldots,y_{n-r}).$$
Let $\lambda_1,\ldots, \lambda_n$ be $n$ indeterminates. Denote by $[n] = \{1,2,\ldots,n\}$. For each subset $I=\{i_1, \ldots, i_r\}\subset [n]$, we denote by \[\lambda_I = (\lambda_{i_1},\ldots,\lambda_{i_r}) \text{ and } I^c=[n]\setminus I.\]
\begin{proposition}\label{doublysympo}
Let $P(x_1,\ldots,x_r,y_1,\ldots,y_{n-r}) \in \mathbb Q[x_1,\ldots,x_r,y_1,\ldots,y_{n-r}]$ be a doubly symmetric polynomial of degree not greater than $r(n-r)$. Then
$$\sum_{I\subset [n],\#I=r}\frac{P(\lambda_I,\lambda_{I^c})}{\prod_{i\in I, j\in I^c}(\lambda_i - \lambda_j)}=\frac{d(r,n)}{r!(n-r)!},$$
where $d(r,n)$ is the coefficient of the monomial $x_1^{n-1}\cdots x_r^{n-1}y_1^{n-1}\cdots y_{n-r}^{n-1}$ in the polynomial 
\[P(x_1,\ldots,x_r,y_1,\ldots,y_{n-r})\prod_{j \neq i}(x_i-x_j)\prod_{j \neq i}(y_i-y_j)\prod_{i=1}^{n-r}\prod_{j=1}^r(y_i - x_j).\]
\end{proposition}
\begin{proof}
We apply Lemma \ref{Lem2} to the following polynomials 
\[F(x_1,\ldots,x_r,y_1,\ldots,y_{n-r}) = P(x_1,\ldots,x_r,y_1,\ldots,y_{n-r})\prod_{j \neq i}(x_i-x_j)\prod_{j \neq i}(y_i-y_j)\prod_{i=1}^{n-r}\prod_{j=1}^r(y_i - x_j),\]
and 
$$Q_1(x)=\cdots =Q_n(x) = Q(x) = \prod_{i=1}^n (x-\lambda_i).$$
By the assumption, $F(x_1,\ldots,x_r,y_1,\ldots,y_{n-r})$ is a doubly symmetric polynomial of degree not greater than $n(n-1)$ and $Q(x)$ is a monic polynomial of degree $n$ with distinct roots $\lambda_1\ldots, \lambda_n$. In this case, the expression \eqref{eq1} can be reduced to
\begin{equation}\label{eq2}
    r!(n-r)!\sum_{I\subset [n],\# I =r}\frac{F(\lambda_I,\lambda_{I^c})}{\prod_{i=1}^n Q'(\lambda_i)}.
\end{equation}
For each $I \subset [n],\# I =r$, we have
\begin{equation*}
    \begin{split}
        F(\lambda_I,\lambda_{I^c}) &= P(\lambda_I,\lambda_{I^c})\prod_{i,j\in I, j\neq i}(\lambda_i-\lambda_j)\prod_{i,j\in I^c, j\neq i}(\lambda_i-\lambda_j)\prod_{i \in I^c,j\in I}(\lambda_i-\lambda_j)\\  &= P(\lambda_I,\lambda_{I^c})\prod_{i,j\in I, j\neq i}(\lambda_i-\lambda_j)\prod_{i\in I^c, j\neq i}(\lambda_i-\lambda_j),
    \end{split}
\end{equation*}
and
\begin{equation*}
    \begin{split}
        \prod_{i=1}^n Q'(\lambda_i) &=  \prod_{j\neq i}(\lambda_i -\lambda_j) \\ &= \prod_{i\in I,j\neq i}(\lambda_i-\lambda_j)\prod_{i\in I^c,j\neq i}(\lambda_i-\lambda_j) \\ &= \prod_{i\in I, j\in I^c}(\lambda_i - \lambda_j)\prod_{i,j\in I, j\neq i}(\lambda_i-\lambda_j)\prod_{i\in I^c, j\neq i}(\lambda_i-\lambda_j).
    \end{split}
\end{equation*}
Then, the expression \eqref{eq2} can be written as
$$r!(n-r)!\sum_{I\subset [n],\# I =r}\frac{P(\lambda_I,\lambda_{I^c})}{\prod_{i\in I, j\in I^c}(\lambda_i - \lambda_j)}.$$
The identity follows.
\end{proof}

\begin{proof}[Proof of Theorem \ref{main}] As already mentioned,  Graf von Bothmer and Ranestad \cite[Proposition 4.1]{BR2009} showed that the algebraic degree $\delta(m,n,r)$ can be computed as
an intersection number on the Grassmannian $G(r,n)$, that is
\begin{equation}\label{schubert}
  \delta(m,n,r) = \int_{G(r,n)}s_{\mathscr l}(S^2U^*)s_{\mathscr k}(S^2Q),
\end{equation}
where $U$ and $Q$ are respectively the universal sub-bundle and quotient bundle on the Grassmannian $G(r,n)$, $S^2Q$ and $S^2U^*$ are respectively the second symmetric power of $Q$ and the dual of $U$, and $s_i(E)$ is the $i$-th Segre class of the dual of the vector bundle $E$.

By the formula \eqref{schubert} and the Bott residue formula in equivariant cohomology (see \cite[Proposition 9.1.5]{CK}), Hiep (see \cite[Theorem 1]{H1}) showed that
\begin{equation}\label{Hiep}
    \delta(m,n,r) = (-1)^{\mathscr l}\sum_{I\subset [n],\# I =r}\frac{A_{\mathscr l,I}A_{\mathscr k,I^c}}{T_I},
\end{equation}
where 
$$T_I = (-1)^{r(n-r)}\prod_{i\in I,j\in I^c} (\lambda_i - \lambda_j),$$
$$A_{\mathscr l,I} = \det\left(\begin{matrix}
e_{1}(\Lambda_I) & e_{2}(\Lambda_I) & e_{3}(\Lambda_I) & \cdots & e_{\mathscr l}(\Lambda_I)\\
1 & e_{1}(\Lambda_I) & e_{2}(\Lambda_I) & \cdots & e_{\mathscr l-1}(\Lambda_I)\\
0 & 1 & e_{1}(\Lambda_I) & \cdots & e_{\mathscr l-2}(\Lambda_I)\\
\vdots & \vdots & \vdots & \ddots & \vdots \\
0 & 0 & 0 & \cdots & e_{1}(\Lambda_I)
\end{matrix}\right)_{\mathscr l \times \mathscr l},$$
$$A_{\mathscr k,I^c} = \det\left(\begin{matrix}
e_{1}(\Lambda_{I^c}) & e_{2}(\Lambda_{I^c}) & e_{3}(\Lambda_{I^c}) & \cdots & e_{\mathscr k}(\Lambda_{I^c})\\
1 & e_{1}(\Lambda_{I^c}) & e_{2}(\Lambda_{I^c}) & \cdots & e_{\mathscr k-1}(\Lambda_{I^c})\\
0 & 1 & e_{1}(\Lambda_{I^c}) & \cdots & e_{\mathscr k-2}(\Lambda_{I^c})\\
\vdots & \vdots & \vdots & \ddots & \vdots \\
0 & 0 & 0 & \cdots & e_{1}(\Lambda_{I^c})
\end{matrix}\right)_{\mathscr k \times \mathscr k}.$$
Note that $e_{i}(\Lambda_I)$ is the $i$-th elementary symmetric polynomial in variables which are the elements of the set 
\[\Lambda_I = \{\lambda_i+\lambda_j \mid i,j \in I, i \leq j\}.\]
By the Jacobi-Trudi identity (see \cite[Theorem 6.1]{egge2019}), we have 
\[h_{\mathscr l}(\Lambda_I) = A_{\mathscr l,I} \quad \text{ and } \quad h_{\mathscr k}(\Lambda_{I^c}) = A_{\mathscr k,I^c}.\]
Thus, the formula \eqref{Hiep} can be rewritten as  
\[\delta(m,n,r) = (-1)^{\mathscr k} \sum_{I\subset [n],\# I =r}\frac{h_{\mathscr l}(\Lambda_I)h_{\mathscr k}(\Lambda_{I^c})}{\prod_{i\in I, j\in I^c}(\lambda_i - \lambda_j)}.\]
By Proposition \ref{doublysympo}, we get the desired identity.
\end{proof}

\section{Some facts about symmetric polynomials}

This section intends to recall some well-known facts about symmetric polynomials. In addition, we clarify some related results that will be used in the sequel. For complete and elegant treatments of symmetric polynomials, we refer to Macdonald’s book  \cite{macdonald1998symmetric}, Manivel's book \cite{manivel2001symmetric} and Egge's book \cite{egge2019}. In this section, we particularly work over the ring $\mathbb Q[x_1,\ldots,x_r]^{\mathcal{S}_r}$ of symmetric polynomials in $r$ variables $x_1,\ldots, x_r$ with rational coefficients.

Let $d$ be a positive integer, then a \emph{partition of} $d$, denoted by $\lambda \vdash d$, is a finite weakly-decreasing sequence $\lambda= (\lambda_1 \geq \lambda_2 \geq \cdots)$ of non-negative integers, such that $\sum_{i} \lambda_i = d$. The $\lambda_i$ are called the \emph{parts} of the partition. Thus, a partition of $d$ has at most $d$ non-zero parts. The number of non-zero parts is the \emph{length} of $\lambda$, denoted by $l(\lambda)$. The sum of all (non-zero) parts of $\lambda$ is the \emph{weight} of $\lambda$, denoted by $|\lambda|$, and clearly $|\lambda| = d$ if $\lambda$ is a partition of $d$. We shall find it convenient not to distinguish between two partitions of the same weight which differ only by a string of zeros at the end. We often use a notation that makes explicit the number of times that a particular integer occurs as a non-zero part. Thus if $\lambda =  (\lambda_1,\lambda_2,\ldots, \lambda_s)$ is a partition of $d$ with length $s$, we sometimes write
$\lambda = (k^{\alpha_k}, \ldots, 2^{\alpha_2}, 1^{\alpha_1})$, where $\alpha_i$ is the number of $i$ in the partition. Note that $\sum_{i = 1}^{k}i\alpha_i = d$.

For each partition $\lambda = (\lambda_1,\ldots,\lambda_r)$, the Schur polynomial is defined to be 
\[s_{\lambda}(x_1,\ldots,x_r) = \frac{a_{\lambda+\delta_r}}{a_{\delta_r}},\]
where $\delta_r = (r-1,\ldots,1,0), a_{\lambda + \delta_r} = \det(x_i^{\lambda_j + r - j})_{r \times r}$, and $a_{\delta_r} = \det(x_i^{r-j})_{r\times r}$. Note that the latter determinant is the well-known Vandermonde determinant, that is
\[a_{\delta_r} = \prod_{1\leq i < j \leq r}(x_i-x_j).\]
The Schur polynomial $s_\lambda$ is in fact a homogeneous symmetric polynomial of degree $|\lambda|$, see \cite{macdonald1998symmetric}. In particular, if $\lambda = (k)$ and $\lambda = (1^k)$ then the Schur polynomials 
\[s_{(k)}(x_1,\ldots,x_r) = h_k(x_1,\ldots,x_r) = \sum_{1 \leq i_1 \leq i_2 \leq \cdots \leq i_k \leq r}x_{i_1}x_{i_2}\ldots x_{i_k},\]
\[s_{(1^k)}(x_1,\ldots,x_r) = e_k(x_1,\ldots,x_r) = \sum_{1 \leq i_1 < i_2 < \cdots < i_k \leq r}x_{i_1}x_{i_2}\ldots x_{i_k},\]
which are respectively called the $k$-th {\it complete homogeneous symmetric polynomial}
and the $k$-th {\it elementary symmetric polynomial}.

The following lemma is useful in sequel. 

\begin{lemma}\label{schur-vandermonde_square}
Let $\lambda$ be a partition of length $l(\lambda) \leq r$. Let $n$ be a positive integer such that $n>r$. Then, the coefficient of the monomial $x_1^{n-1}\ldots x_r^{n-1}$ in the polynomial
\begin{equation}\label{schur-vandermonde_square-eq}
s_{\lambda}(x_1,\ldots,x_r)\prod_{j\neq i}(x_i - x_j)
\end{equation}
is equal to $r!$ if $\lambda = ((n-r)^r)$, and $0$ if $\lambda \neq ((n-r)^r)$.
\end{lemma}
\begin{proof}
We first prove that if $\lambda \neq ((n-r)^r)$, then the polynomial \eqref{schur-vandermonde_square-eq} has no the monomial $x_1^{n-1}\cdots x_r^{n-1}$ in its expansion. Indeed, we have
\begin{eqnarray*}
s_{\lambda}(x_1,\ldots,x_r)\prod_{j\neq i}(x_i - x_j) & = & (-1)^{\frac{r(r+1)}{2}}a_{\lambda + \delta_r}(x_1,\ldots,x_r)a_{\delta_r}(x_1,\ldots,x_r)\\
& = & (-1)^{\frac{r(r+1)}{2}} \det(x_i^{\lambda_j+r-j})_{r\times r}\det(x_i^{r-j})_{r\times r}\\
& = & (-1)^{\frac{r(r+1)}{2}} \left( \sum_{\sigma\in \mathcal{S}_r}\sgn(\sigma)x^{\sigma(\lambda+\delta_r)} \right) \left( \sum_{\theta\in \mathcal{S}_r}\sgn(\theta)x^{\theta(\delta_r)} \right)\\
& = & (-1)^{\frac{r(r+1)}{2}} \left( \sum_{\sigma,\theta \in \mathcal{S}_r}\sgn(\sigma)\sgn(\theta)x^{\sigma(\lambda+\delta_r) + \theta(\delta_r)} \right).
\end{eqnarray*}
Suppose that $x_1^{n-1}\cdots x_r^{n-1}$ is a term of the polynomial \eqref{schur-vandermonde_square-eq}, then there exist a partition $\lambda = (\lambda_1, \ldots,\lambda_r)$ such that $\sigma(\lambda+\delta_r) + \theta(\delta_r) = ((n-1)^r)$ for some $\sigma,\theta \in \mathcal{S}_r$. Since $\delta_r =(r-1,r-2, \ldots,1,0)$, it follows $\sigma(\lambda+\delta_r)$ is a permutation of the set $\{n-r,\ldots, n-1\}$. Thus, the maximal (resp. minimal) term in $\sigma(\lambda+\delta_r)$, that is $\lambda_1 +r-1$ (resp. $\lambda_r$), must be $
n-1$ (resp. $n-r$). Hence, $\lambda_1=n-r$ and $\lambda_r = n-r$. Thus $\lambda = ((n-r)^r)$, a contradiction.

Now, let us suppose that $\lambda = ((n-r)^r)$. It is obvious that
\begin{eqnarray*}
s_{((n-r)^r)}(x_1,\ldots,x_r)\prod_{j\neq i}(x_i - x_j) & = & e^{n-r}_r(x_1,\ldots,x_r)\prod_{j\neq i}(x_i - x_j)\\
& = & \prod_{i=1}^r x^{n-r}_i \prod_{j\neq i}(x_i - x_j).
\end{eqnarray*}
It turns out that the coefficient of the monomial $x_1^{n-1} \ldots x_r^{n-1}$ in the polynomial \eqref{schur-vandermonde_square-eq} in corresponding to $\lambda = ((n-r)^r)$ is the coefficient of the monomial $x_1^{r-1}\ldots x_r^{r-1}$ in $\prod_{j\neq i}(x_i - x_j)$. Moreover, it was proved in \cite{Ze} that $r!$ is the coefficient of the monomial $x_1^{r-1}\ldots x_r^{r-1}$ in $\prod_{j\neq i}(x_i - x_j)$.
\end{proof}

Consider the infinite Pascal triangle matrix
\[P = \left(\begin{array}{cccccccccc}
1          & 0         & 0         & 0                   & 0                                 &\ldots    \\
1          & 1         & 0         & 0                   & 0                                 &\ldots    \\
1          & 2         & 1         & 0                   & 0                                 &\ldots    \\
1          & 3         & 3         & 1                   & 0                                 &\ldots    \\
1          & 4         & 6         & 4                   & 1                                 &\ldots    \\
\vdots & \vdots & \vdots & \vdots           & \vdots                        & \ddots   \\
\end{array}\right),\]
where $P_{ij} = \binom{i}{j}$. For every pair of finite subset $I,J \subset \mathbb N$, denote by $M_{IJ}$ the submatrix of $P$ with rows indexed by $I$ and columns indexed by $J$. For each finite subset $I = \{i_1,\ldots,i_r\} \subset \mathbb N$, we denote
\begin{equation}\label{psi_deter}
\psi_I = \sum_{J\subset \mathbb N, \#J=r} \det(M_{I,J}),
\end{equation}
and
\begin{equation}
\lambda(I) = (i_r - (r-1),\ldots, i_1).
\end{equation}
The next lemma is a consequence of  Proposition A.15 in \cite{LLT}.
\begin{lemma}\label{Prop.A.15} Let $\mathcal{X}$ be the set mentioned in \eqref{sum_variables} and $d$ a positive integer. Then, we have
\begin{equation}\label{complete_pol}
h_d(\mathcal{X})=\sum_{I \subset \mathbb N, \#I =r, |\lambda(I)|=d}\psi_Is_{\lambda(I)}(x_1,\ldots,x_r).
\end{equation}
\end{lemma}
\begin{example}
For $d = 2, r = 2$, we have
\begin{align*}
    h_2(2x_1,x_1+x_2,2x_2) &= 7x_1^2 + 10x_1x_2 + 7x_2^2\\
    &= 7s_{(2)} + 3s_{(1^2)}.
\end{align*}
\end{example}

Using above lemmas, we prove the following results.

\begin{proposition}\label{key_lm}  
Let $\mathcal{X}$ be the set mentioned in \eqref{sum_variables}. Then, the coefficient of the monomial $x_1^{r}\cdots x_{r}^{r}$ in the polynomial
\begin{equation}\label{e_h_V1}
e_{k}(x_1,\ldots,x_{r})h_{r-k}(\mathcal{X})\prod_{j \neq i}(x_i-x_j)
\end{equation}
is equal to $\binom{r+1}{k+1}r!$.
\end{proposition}

\begin{proof} 
Let us set $x=(x_1,\ldots,x_{r})$. Then, we have
\begin{eqnarray*}
  e_{k}(x)h_{r-k}(\mathcal{X}) &=& e_{k}(x)\sum_{I \subset \mathbb N, \# I =r, |\lambda(I)| = r-k}\psi_Is_{\lambda(I)}(x)\\
  &=& \sum_{I \subset \mathbb N, \# I =r, |\lambda(I)| = r-k}\psi_I\big(s_{\lambda(I)}(x)e_{k}(x)\big)\\
  &=& \sum_{I \subset \mathbb N, \# I =r, |\lambda(I)| = r-k}\psi_I\bigg(\sum_{\gamma}s_{\gamma}(x)\bigg).
\end{eqnarray*}
The latter equality is due to the Pieri rule (see, for example, \cite[Proposition 1.2.5]{manivel2001symmetric}). Moreover, by Lemma \ref{schur-vandermonde_square}, one has the coefficient of the monomial $x_1^{r}\cdots x_{r}^{r}$ in the polynomial \eqref{e_h_V1} must be $\psi_{I} r!$ with $I = [[r]] \setminus \{k\} = \{0,1,\ldots,k-1,k+1,\ldots,r-1,r\}$. It remains to prove that 
\[\psi_{I} = \binom{r+1}{k+1}.\]
Obviously, $I$ is obtained by removing an element from the set $[[r]]$, then only those sets $J$, in expression \eqref{psi_deter}, which are subsets of $[[r]]$ will make $\det(M_{I,J}) \neq 0$, otherwise $\det(M_{I,J}) =0$. In particular, we have $I=[[r]] \setminus \{k\}$ and 
$$
\begin{cases}
\textrm{if } J \nsubset [[r]], \textrm{ then } \det(M_{I,J}) = 0,\\
\textrm{if } J = [[r]]\setminus \{a\} \textrm{ for some } a<k, \textrm{ then } \det(M_{I,J}) = 0,\\
\textrm{if } J = [[r]]\setminus \{b\} \textrm{ for some } b\geq k, \textrm{ then } \det(M_{I,J}) = \binom{b}{k}.\\
\end{cases}
$$
It follows
\begin{eqnarray*}
\psi_I = \sum_{J\subset \mathbb N,\#J=r} \det(M_{I,J})
= \sum_{b=k}^{r}\binom{b}{k}
= \binom{r+1}{k+1}.
\end{eqnarray*}
The last equality is due to the hockey-stick identity. 
\end{proof}

\begin{proposition}\label{key_lm2}
Let $\mathcal{X}$ be the set mentioned in \eqref{sum_variables}. Then, the coefficient of the monomial $x_1^{r+1}\cdots x_{r}^{r+1}$ in the polynomial
\begin{equation}\label{eG}
e_k(x_1,\ldots,x_r)h_{2r-k}(\mathcal{X})\prod_{j \neq i}(x_i-x_j)
\end{equation}
is equal to $(k+1)\binom{r+3}{k+3}r!$.
\end{proposition}\label{lemGen}
\begin{proof}
Let $x = (x_1,\ldots,x_{r})$. We have
\begin{eqnarray*}
e_k(x)h_{2r-k}(\mathcal{X}) &=& e_k(x)\sum_{I \subset \mathbb N, \#I=r,|\lambda(I)|= 2r-k}\psi_Is_{\lambda(I)}(x)\\
&=& \sum_{I \subset \mathbb N, \#I=r,|\lambda(I)|= 2r-k}\psi_I\big(s_{\lambda(I)}(x)e_k(x)\big)\\
&=& \sum_{I \subset \mathbb N, \#I=r,|\lambda(I)|= 2r-k}\psi_I\bigg(\sum_{\gamma}s_{\gamma}(x)\bigg).
\end{eqnarray*}
The latter equality is due to the Pieri rule. By Lemma \ref{schur-vandermonde_square}, the coefficient of the monomial $x_1^{r+1}\cdots x_{r}^{r+1}$ in the polynomial (\ref{eG}) must be $\psi_{I}r!$ with $\lambda(I) = (2^{r-k},1^k)$, that is
\[I = (1,\ldots,k,k+2,\ldots,r+1).\]
For each $J = [[r+1]]\setminus \{i,j\}$ with $0\leq i < j\leq r+1$, we have
\[\det(M_{I,J}) = \binom{j}{k+1}-\binom{i}{k+1}.\]
Thus,
\begin{align*}
    \psi_{I} &= \sum_{J\subset \mathbb N, \#J=r}\det(M_{I,J})\\
    &= \sum_{0\leq i<j\leq r+1}\left(\binom{j}{k+1}-\binom{i}{k+1}\right)\\
    &= \sum_{j=1}^{r+1}\sum_{i=0}^{j-1}\left(\binom{j}{k+1}-\binom{i}{k+1}\right)\\
    &=  (k+1)\sum_{j=1}^{r+1}\binom{j+1}{k+2}\\
    &= (k+1)\binom{r+3}{k+3}.
\end{align*}
The last equality is due to the hockey-stick identity.
\end{proof}

\section{Examples and applications}
\begin{proposition}[cf. Proposition 9 in \cite{NRS}]\label{duality_relation}
The algebraic degree of semidefinite programming, $\delta(m,n,r)$, satisfies the duality relation
\begin{equation}\label{dual_form}
    \delta(m,n,r) = \delta \left(\binom{n+1}{2}-m,n,n-r \right).
\end{equation}
\end{proposition}
\begin{proof}
By Theorem \ref{main}, we have 
\begin{equation*}
    \delta\left(\binom{n+1}{2}-m,n,n-r\right) = (-1)^{\mathscr l}\frac{c\left(\binom{n+1}{2}-m,n,n-r\right)}{r!(n-r)!},
\end{equation*}
where $c\left(\binom{n+1}{2}-m,n,n-r\right)$ is the coefficient of the monomial $y_1^{n-1}\ldots y_{n-r}^{n-1}x_1^{n-1}\ldots x_r^{n-1}$ in the polynomial
\[h_{\mathscr k}(\mathcal Y)h_{\mathscr l}(\mathcal X)\prod_{j \neq i}^{n-r}(y_i-y_j)\prod_{j \neq i}^{r}(x_i-x_j)\prod_{i=1}^{r}\prod_{j=1}^{n-r}(x_i - y_j).\]
Note that 
\[\prod_{i=1}^{r}\prod_{j=1}^{n-r}(x_i - y_j) = (-1)^{r(n-r)} \prod_{i=1}^{r}\prod_{j=1}^{n-r}(y_j - x_i),\]
and 
\[\mathscr l + \mathscr k = r(n-r).\]
Thus, we have 
\[c\left(\binom{n+1}{2}-m,n,n-r\right) = (-1)^{\mathscr l + \mathscr k}c(m,n,r).\]
The proposition is proved.
\end{proof}

\begin{proposition}[cf. Theorem 11.1 in \cite{NRS}]\label{n-1} When the optimal rank $r$ equals $n-1$, the algebraic degree of semidefinite programming is
$$\delta(m,n,n-1) = 2^{m-1}\binom{n}{m}.$$
\end{proposition}
\begin{proof}
By Theorem \ref{main}, one has
$$\delta(m,n,n-1) = (-1)^{m-1}\frac{c(m,n,n-1)}{(n-1)!},$$
where $c(m,n,n-1)$ is the coefficient of the monomial $x_1^{n-1}\cdots x_{n-1}^{n-1}y_1^{n-1}$ in the following polynomial
\begin{equation*}
    \begin{split}
        & h_{n-m}(\mathcal{X})h_{m-1}(2y_1)\prod_{j \neq i}^{n-1}(x_i-x_j)\prod_{i=1}^{n-1}(y_1-x_i) \\
        & = h_{n-m}(\mathcal{X})(2y_1)^{m-1}\prod_{j \neq i}^{n-1}(x_i-x_j)\left( \sum_{k=0}^{n-1}(-1)^ke_k(x_1,\ldots,x_{n-1})y_1^{n-1-k} \right) \\
        & = (-1)^{m-1}2^{m-1}e_{m-1}(x_1,\ldots,x_{n-1})h_{n-m}(\mathcal{X})\prod_{j \neq i}^{n-1}(x_i-x_j)y_1^{n-1}+\textrm{ other terms which} \\
       & \hspace{10.5 cm} \textrm{do not contain factor $y_1^{n-1}$}\\
&= (-1)^{m-1}2^{m-1}\binom{n}{m}(n-1)! x_1^{n-1}\cdots x_{n-1}^{n-1}y_1^{n-1}+\textrm{other different terms}.
    \end{split}
\end{equation*}
The last equality follows from Proposition \ref{key_lm} with $k = m-1$.
\end{proof}

\begin{proposition}[cf. Theorem 11.2 in \cite{NRS}]\label{3_n-2} We have
$$\delta(3,n,n-2) = \binom{n+1}{3}.$$
\end{proposition}

\begin{proof}
By Theorem \ref{main}, we have
$$\delta(3,n,n-2) = \frac{c(3,n,n-2)}{(n-2)!2!},$$
where $c(3,n,n-2)$ is the coefficient of the monomial $x_1^{n-1}\cdots x_{n-2}^{n-1}y_1^{n-1}y_2^{n-1}$ in the following polynomial
\begin{equation*}
    \begin{split}
        & h_{2n-4}(\mathcal{X})\prod_{j \neq i}(x_i-x_j)(2y_1y_2 - y_1^2 - y_2^2)\prod_{i=1}^{n-2}(y_1-x_i)\prod_{i=1}^{n-2}(y_2-x_i) \\
        & = 2h_{2n-4}(\mathcal{X})\prod_{j \neq i}(x_i-x_j)y_1^{n-1}y_2^{n-1}+\textrm{ other terms which do not contain factor $y_1^{n-1}y_2^{n-1}$} \\
        &= 2\binom{n+1}{3}(n-2)! x_1^{n-1}\cdots x_{n-1}^{n-1}y_1^{n-1}y_2^{n-1}+\textrm{other different terms}.
    \end{split}
\end{equation*}
The last equality follows from Proposition \ref{key_lm2} with $k=0$.
\end{proof}

\begin{proposition}[cf. Theorem 11.2 in \cite{NRS}]\label{4_n-2} We have
$$\delta(4,n,n-2) = 6\binom{n+1}{4}.$$
\end{proposition}
\begin{proof}
By Theorem \ref{main}, we have
$$\delta(4,n,n-2) = -\frac{c(4,n,n-2)}{(n-2)!2!},$$
where $c(4,n,n-2)$ is the coefficient of the monomial $x_1^{n-1}\cdots x_{n-2}^{n-1}y_1^{n-1}y_2^{n-1}$ in the following polynomial
\begin{equation*}
    \begin{split}
       & h_{2n-5}(\mathcal{X})h_1(\mathcal Y)\prod_{j \neq i}(x_i-x_j)(2y_1y_2 - y_1^2 - y_2^2)\prod_{i=1}^{n-2}(y_1-x_i)\prod_{i=1}^{n-2}(y_2-x_i) \\
        & = h_{2n-5}(\mathcal{X})h_1(\mathcal Y)\prod_{j \neq i}(x_i-x_j)(2y_1y_2 - y_1^2 - y_2^2)\prod_{i=1}^{n-2}(y_1-x_i)\prod_{i=1}^{n-2}(y_2-x_i) \\
        & = -6e_1(x) h_{2n-5}(\mathcal{X})\prod_{j \neq i}(x_i-x_j)y_1^{n-1}y_2^{n-1}+\textrm{ other terms which} \\
       & \hspace{10.5 cm} \textrm{do not contain factor $y_1^{n-1}y_2^{n-1}$}\\
        &= -12\binom{n+1}{4}(n-2)! x_1^{n-1}\cdots x_{n-1}^{n-1}y_1^{n-1}y_2^{n-1}+\textrm{other different terms}.
    \end{split}
\end{equation*}
The last equality follows from Proposition \ref{key_lm2} with $k=1$.
\end{proof}

\subsection*{Acknowledgements}
This research is supported by the 2021 Annual Research Program of the Dalat University. The first author is partially funded by the Simons Foundation Grant Targeted for Institute of Mathematics, Vietnam Academy of Science and Technology. He would like to thank the institute for the very kind support and hospitality. 

\bibliography{references}{}
\bibliographystyle{alpha}

\end{document}